\documentclass[12 points]{article}
\usepackage{amssymb}
\usepackage{eucal}
\usepackage{amsmath}
\usepackage{amsthm}

\numberwithin{equation}{section}
\newtheorem{thm}{\bf Theorem}[section]
\newtheorem{lem}[thm]{\bf Lemma}
\newtheorem{prop}[thm]{\bf Proposition}

\newtheorem{defn}{\bf Definition}[section]

\theoremstyle{remark}
\newtheorem{rem}{\bf Remark}[section]

\begin{document}

\title{\large\bf GEOMETRICAL DISSIPATION FOR DYNAMICAL SYSTEMS}
\author{Petre Birtea* and Dan Com\u{a}nescu*\\
{\small *Department of Mathematics, West University of Timi\c soara}\\
{\small Bd. V. P\^ arvan, No 4, 300223 Timi\c soara, Rom\^ ania}\\
{\small birtea@math.uvt.ro, comanescu@math.uvt.ro}\\
}
\date{ }
\maketitle

\begin{abstract}
On a Riemannian manifold $(M,g)$ we consider the $k+1$ functions $F_1,...,F_k,G$ and construct the vector fields that conserve $F_1,...,F_k$ and dissipate $G$ with a prescribed rate. We study the geometry of these vector fields and prove that they are of gradient type on regular leaves corresponding to $F_1,...,F_k$. By using these constructions we show that the cubic Morrison dissipation and the Landau-Lifschitz equation can be formulated in a unitary form.
\end{abstract}



\section{Introduction}

An isolated dynamical system is a conservative system in the sense that there exist certain conserved quantities. Such systems are often described in a Hamilton-Poisson setting for which the energy and the Casimir functions are conserved quantities. In real life certain dynamical parameters are not conserved due to the fact that the systems are not completely isolated. Another instance when we have a dissipative behavior of some parameters is when control terms are added. For a large class of dissipative terms various mathematical formulations have been found.

Starting from the notion of Poisson bracket, a dissipative bracket was introduced in the study of dissipative Hamilton-Poisson systems by M. Grmela \cite{grmela-84}, A. Kaufman \cite{kaufman-84}, and P.J. Morrison \cite{morrison-86}. In \cite{morrison-86}, P.J. Morrison coined the notion of metriplectic
systems which are Hamilton-Poisson systems that are perturbed with a dissipation of
metric type. Dissipative terms and their implications
for dynamics have also been studied in connection to various dynamical
systems derived from mathematical physics, see \cite{birtea-comanescu}, \cite{bloch-marsden-ratiu}, \cite{grmela-86}, \cite{kaufman-85}, \cite{kaufman-morrison}, \cite{kaufman-turski}, \cite{morrison-84}.

Another type of dissipation which is called double bracket dissipation, was introduced by Brockett, see \cite{brockett-91} and \cite{brockett-93}. The double bracket equation is defined in a Lie algebra setting and has an important role in the study of various completely integrable systems, see \cite{bloch-90}, \cite{bloch-brockett-crouch}, \cite{bloch-brockett-ratiu}, \cite{bloch-flaschka-ratiu}. It has been shown that this flow is a gradient flow on the
adjoint orbits, see \cite{bloch-brockett-ratiu}. The metric considered is the "standard" or "normal" metric, see \cite{atiyah} and \cite{besse}. A particular example of such dissipation is given by Landau-Lifschitz equation.

We show in the current paper that the cubic Morrison dissipation and the Landau-Lifschitz equation can be formulated in an unitary form. In Section 2 we start with the $k+1$ functions $F_1,...,F_k,G$ and construct the vector fields that conserve $F_1,...,F_k$ and the function $G$ oscillates along these vector fields with a prescribed rate. We prove that all these control vector fields are generated by a vector field that we will call the standard control vector field. We apply this construction to the case when we have a dynamical system which has $k+1$ conserved quantities. We construct a perturbation which dissipates one of the conserved quantities and conserves the remanning $k$ of them.

In the Euclidean case such a dissipation was constructed in \cite{indian} using exterior algebra. In Section 3 we extend this construction to a general Riemannian manifold and moreover, we prove that this dissipation is minus the standard control vector field constructed in Section 2. This generalization allows the study of dissipative models which have as phase space a general Riemannian manifold.

In Section 4 we study the geometry of the standard control vector field. In analogy with the case of double bracket dissipation, see \cite{bloch-brockett-ratiu}, we prove that when restricted to regular leaves of the function $\mathbf{F}=(F_1,...,F_k)$, the standard control vector field has a gradient formulation. On every regular leaf we construct a certain Riemannian metric such that the restricted standard control vector field is a gradient vector field with respect to this metric. First we construct a symmetric contravariant 2-tensor that is degenerate on the phase space. We prove that this tensor is in fact nondegenerate when restricted to regular leaves. The Riemannian metric on the regular leaves will be the inverse of the restriction of this tensor. Moreover, we prove that this metric is a conformal metric with the first fundamental form of the regular leaves.

In Section 5 we prove that the standard control vector field can also be written as a scaled orthogonal projection on the regular leaves of the gradient vector field $\nabla G$. Using this formulation, we study the connection between the standard control vector field generated by sets of functionally dependent conserved quantities.

In Section 6 we prove that the Landau-Lifschitz equation can be regarded as a perturbed system with the perturbation being a standard control vector field. This perturbation was formulated before as a double bracket dissipation, see \cite{marsden-92}, \cite{bloch-krishnaprasad-marsden-ratiu}, \cite{ortega-planas}.
We will also show that the cubic dissipation of the rigid body introduced by Morrison \cite{morrison-86} is again of the form given by a standard control vector field. Both systems can be described as perturbed Hamilton-Poisson systems on the Lie algebra $so(3)$. Double bracket dissipation is obtained by conserving the Casimir function and dissipating the Hamiltonian function. Morrison dissipation is obtained by conserving the Hamiltonian function and dissipating the Casimir function.

Various formulas and notations that are used throughout this paper are listed in the Appendix.

\section{Construction of the dissipation}

Having the $k+1$ functions $F_1,...,F_k,G$, we construct in this section a family of vector fields that conserves $F_1,...,F_k$ and the function $G$ oscillates along these vector fields with a prescribed rate.

Let $(M,g)$  be a Riemannian manifold and $F_1,...,F_k,G:M\rightarrow \mathbb{R}$ be $k+1$ smooth functions. We construct a vector field $\mathbf{u}\in \mathcal{X}(M)$ that conserves $F_1,...,F_k$ and dissipates $G$.
A function $F_{\alpha}$ is a conserved quantity if and only if
$$0=L_XF_{\alpha}=\frac{\partial F_{\alpha}}{\partial
x^i}X^i=\delta_{ia}\frac{\partial F_{\alpha}}{\partial
x^a}X^i=g_{ij}g^{aj}\frac{\partial F_{\alpha}}{\partial
x^a}X^i=g(X^i\frac{\partial}{\partial x^i},g^{aj}\frac{\partial
F_{\alpha}}{\partial x^a}\frac{\partial}{\partial x^j})=<X,\nabla
F_{\alpha}>,$$
where $<\cdot,\cdot>$ is the scalar product
generated by the metric $g$ and $\nabla
F_{\alpha}=g^{aj}\frac{\partial F_{\alpha}}{\partial
x^a}\frac{\partial}{\partial x^j}$ is the gradient vector field
generated by the function $F_{\alpha}$ on the Riemannian manifold
$(M,g)$.
The function $G$ oscillates along the solutions of the vector field $\mathbf{u}$ after the rule
\begin{equation*}
\frac{dG}{dt}(x(t,x_0))=h(x(t,x_0)),
\end{equation*}
where the rate of dissipation $h:M\rightarrow \mathbb{R}$ is a smooth function and $x(\cdot, x_0)$ is the solution of the system $\dot{x}=\mathbf{u}(x)$ with the initial condition $x_0$. We can write the above statements equivalently
\begin{equation}\label{problema}
    \left\{%
\begin{array}{ll}
    <\mathbf{u},\nabla F_{1}>=0 \\
   \,\,\,\,\,\,\,\,\, ... \\
   <\mathbf{u},\nabla F_{k}>=0 \\
   <\mathbf{u},\nabla G_{\,}>=h. \\
\end{array}%
\right.
\end{equation}

For $x\in M$, we can decompose the vector $\mathbf{u}(x)=\mathbf{v}(x)+\mathbf{w}(x)$, where
$\mathbf{v}(x)\in Sp[\nabla F_1(x),...,\nabla F_k(x),\nabla G(x)]$ is the dissipative part and
$\mathbf{w}(x)\in Sp^{\perp}[\nabla F_1(x),...,\nabla F_k(x),\nabla G(x)]$ is the conservative part.
Consequently, there exists $\alpha_1(x),...,\alpha_k(x),\alpha
(x)\in \mathbb{R}$ such that
\begin{equation}\label{forma-v}
\mathbf{v}(x)=\sum_{i=1}^k\alpha_i(x)\nabla F_i(x)+\alpha (x)\nabla
G(x).
\end{equation}
If $\nabla F_1(x),...,\nabla F_k(x),\nabla G(x)$ are linear
independent, then locally around $x\in M$ the functions
$\alpha_1,...,\alpha_k,\alpha$ can be chosen to be smooth and they
are unique with this property.

The algebraic system \eqref{problema} becomes
\begin{equation}\label{problemah}
    \left\{%
\begin{array}{ll}
    \sum_{i=1}^k\alpha_i(x)<\nabla F_i(x),\nabla F_{1}(x)>+\alpha (x) <\nabla G(x),\nabla F_{1}(x)>=0 \\
   \,\,\,\,\,\,\,\,\, ... \\
    \sum_{i=1}^k\alpha_i(x)<\nabla F_i(x),\nabla F_{k}(x)>+\alpha (x) <\nabla G(x),\nabla F_{k}(x)>=0 \\
    \sum_{i=1}^k\alpha_i(x)<\nabla F_i(x),\nabla G(x)\,>+\alpha (x) <\nabla G(x)\,,\nabla G(x)>\,=h(x) \\
\end{array}%
\right.
\end{equation}
In what follows, we will introduce several notations. If $f_1,...,f_r,g_1,...,g_s:M\rightarrow \mathbb{R}$ are smooth functions on the manifold $M$, we define the $r\times s$ matrix
\begin{equation}\label{sigma}
\Sigma_{(g_1,...,g_s)}^{(f_1,...,f_r)}=\left(%
\begin{array}{cccc}
  <\nabla g_1,\nabla f_{1}> & ... & <\nabla g_s,\nabla f_{1}> \\
  ... & ... & ... \\

  <\nabla g_1,\nabla f_r> & ... & <\nabla g_s,\nabla f_r> \\
\end{array}%
\right)
\end{equation}
We solve the linear system \eqref{problemah} for the unknowns
$\alpha_1,...,\alpha_k,\alpha$. The associated matrix is $\Sigma_{(F_1,...,F_k,G)}^{(F_1,...,F_k,G)}$ and the augmented matrix is given by
$$\widetilde{\Sigma}_{(F_1,...,F_k,G)}^{(F_1,...,F_k,G)}=\left(%
\begin{array}{ccccc|c}
   &  &  &  &  & 0 \\
   &  &  &  &  & ... \\
   &  & \Sigma_{(F_1,...,F_k,G)}^{(F_1,...,F_k,G)} &  &  & ... \\
   &  &  &  &  & 0 \\
   &  &  &  &  & h \\
\end{array}%
\right).$$
The determinant of the Gram matrix generated by the vectors $\nabla F_1(x),...,\nabla F_k(x),\nabla G(x)$  has the properties that $\det \Sigma_{(F_1,...,F_k,G)}^{(F_1,...,F_k,G)}(x)\geq 0$ and equality holds when the vectors are linear dependent, see \cite{horn}.
If $\det \Sigma_{(F_1,...,F_k,G)}^{(F_1,...,F_k,G)}(x)\neq 0$, then $\textrm{rank} \Sigma_{(F_1,...,F_k,G)}^{(F_1,...,F_k,G)}(x)=\textrm{rank}\widetilde{\Sigma}_{(F_1,...,F_k,G)}^{(F_1,...,F_k,G)}(x)$.
Consequently, the linear system \eqref{problemah} is compatible and
according to Cramer's rule we obtain the solution
\begin{equation}\label{solutiih}
    \left\{%
\begin{array}{ll}
    \alpha_i(x)=(-1)^{i+k+1}\frac{h(x)}{\det \Sigma_{(F_1,...,F_k,G)}^{(F_1,...,F_k,G)}(x)}\det \Sigma_{(F_1,...,\widehat{F_i},...,F_k,G)}^{(F_1,...,F_k)}(x) \\
    \alpha(x)=\frac{h(x)}{\det \Sigma_{(F_1,...,F_k,G)}^{(F_1,...,F_k,G)}(x)}\det \Sigma_{(F_1,...,F_k)}^{(F_1,...,F_k)}(x) \\
\end{array}%
\right.,
\end{equation}
where $\,\,\widehat{\cdot}\,\,$ represent the missing term.

In the case when $\det \Sigma_{(F_1,...,F_k,G)}^{(F_1,...,F_k,G)}(x)=0$, we will discuss the compatibility of the linear
system \eqref{problemah}. If $\textrm{rank} \Sigma_{(F_1,...,F_k,G)}^{(F_1,...,F_k)}(x)<\textrm{rank} \Sigma_{(F_1,...,F_k,G)}^{(F_1,...,F_k,G)}(x)$, then $\textrm{rank}
\Sigma_{(F_1,...,F_k,G)}^{(F_1,...,F_k,G)}(x)=\textrm{rank} \widetilde{\Sigma}_{(F_1,...,F_k,G)}^{(F_1,...,F_k,G)}(x)$ and the linear system is compatible. If $\textrm{rank} \Sigma_{(F_1,...,F_k,G)}^{(F_1,...,F_k)}(x)=\textrm{rank} \Sigma_{(F_1,...,F_k,G)}^{(F_1,...,F_k,G)}(x)$, then the matrices $\Sigma_{(F_1,...,F_k,G)}^{(F_1,...,F_k)}(x)$ and $\Sigma_{(F_1,...,F_k,G)}^{(F_1,...,F_k,G)}(x)$ have a
common principal minor. The system is compatible if and only if
$h(x)=0$.

On the open set $\Omega:=\{x\in M\,|\,\det\Sigma_{(F_1,...,F_k,G)}^{(F_1,...,F_k,G)}(x)\neq
   0\}$ we can use the solution found in \eqref{solutiih} and write the vector field in \eqref{forma-v} as $\mathbf{v}=\frac{h}{\det\Sigma_{(F_1,...,F_k,G)}^{(F_1,...,F_k,G)}}\mathbf{v_0}$, where $\mathbf{v_0}\in \mathcal{X}(M)$ is the vector field which we
will call {\bf the standard control vector field} and is given by
\begin{equation}\label{v0}
    \mathbf{v_0}=\sum_{i=1}^k(-1)^{i+k+1}\det \Sigma_{(F_1,...,\widehat{F_i},...,F_k,G)}^{(F_1,...,F_k)}\nabla
    F_i+\det\Sigma_{(F_1,...,F_k)}^{(F_1,...,F_k)}\nabla G.
\end{equation}
For any $x\in M$, it is straightforward to see that the set
$\{\alpha_1^0(x),...,\alpha_k^0(x),\alpha^0(x)\}$, where
$\alpha_i^0(x):=(-1)^{i+k+1}\det \Sigma_{(F_1,...,\widehat{F_i},...,F_k,G)}^{(F_1,...,F_k)}(x)$ and
$\alpha^0(x):=\det\Sigma_{(F_1,...,F_k)}^{(F_1,...,F_k)}(x)$ is a solution of the system
\eqref{problemah} for $h(x)=\det\Sigma_{(F_1,...,F_k,G)}^{(F_1,...,F_k,G)}(x)$. Consequently, the vector field $\mathbf{v_0}$ is a solution for \eqref{problema} in the particular case when we consider $h(x)=\det\Sigma_{(F_1,...,F_k,G)}^{(F_1,...,F_k,G)}(x)$ for any $x\in M$. For the case when we have only two conserved quantities, $F$ and $G$ for the initial system the standard control vector field has the form
$$\mathbf{v}_0=-<\nabla F,\nabla G>\nabla F+||\nabla F||^2\nabla G.$$

As a summary of the above considerations we have the following result.

\begin{thm}\label{caracterizare-control} We have the following characterization of the control
vector field $\mathbf{u}\in \mathcal{X}(M)$ that is a solution of \eqref{problema}.
\begin{itemize}
   \item [(i)] Let $\Omega :=\{x\in M\,|\,\det\Sigma_{(F_1,...,F_k,G)}^{(F_1,...,F_k,G)}(x)\neq
   0\}$ which is an open subset of $M$. Any control
vector field $\mathbf{u}\in \mathcal{X}(\Omega )$ that satisfies
\eqref{problema} (i.e. dissipates the function $G$ with the rate of dissipation $h$) is of the form
$$\mathbf{u}(x)=\frac{h(x)}{\det\Sigma_{(F_1,...,F_k,G)}^{(F_1,...,F_k,G)}(x)}\mathbf{v_0}(x)+\mathbf{w}(x),\,\,\forall x\in \Omega,$$
where $\mathbf{w}\in \mathcal{X}(\Omega )$ with $\mathbf{w}(x)\in Sp^{\perp}[\nabla
F_1(x),...,\nabla F_k(x),\nabla G(x)]$.

  \item [(ii)] If the function $\frac{h}{\det\Sigma_{(F_1,...,F_k,G)}^{(F_1,...,F_k,G)}}:\Omega\rightarrow \mathbb{R}$ can be prolonged to a continuous function $q:M\rightarrow \mathbb{R}$, then the
control vector field has the form $\mathbf{u}(x)=q(x)\mathbf{v_0}(x)+\mathbf{w}(x)$, where $\mathbf{w}\in
\mathcal{X}(M)$ with $\mathbf{w}(x)\in Sp^{\perp}[\nabla F_1(x),...,\nabla
F_k(x),\nabla G(x)]$.

 \item [(iii)] In the particular case when $h=\det\Sigma_{(F_1,...,F_k,G)}^{(F_1,...,F_k,G)}$, the control vector field that is a solution of \eqref{problema} has the form
$$\mathbf{u}(x)=\mathbf{v_0}(x)+\mathbf{w}(x),$$
where $\mathbf{w}\in
\mathcal{X}(M)$, with $\mathbf{w}(x)\in Sp^{\perp}[\nabla F_1(x),...,\nabla
F_k(x),\nabla G(x)]$.
\end{itemize}
\end{thm}
\vspace{2mm}
The dissipative part of the control vector field $\mathbf{u}$ can be written always as $\mathbf{v}=q\mathbf{v}_0$, where $q$ is a smooth function defined at least on the open subset $\Omega$ of $M$. In the paper \cite{indian} is constructed an infinite sequence of high-order dissipative vector fields. As a consequence of the above theorem, we obtain that this sequence is generated by the standard control vector field $\mathbf{v}_0$.

We apply the above construction to the case when we have a dynamical system
\begin{equation}\label{sistem}
    \dot{x}=X(x),
\end{equation}
where $X\in \mathcal{X}(M)$. Suppose \eqref{sistem} admits
$F_1,...,F_k,G:M\rightarrow{\mathbb{R}}$ smooth $k+1$ conserved quantities.
We search for control vector fields $\mathbf{u}\in
\mathcal{X}(M)$ such that the perturbed system
\begin{equation}\label{perturbat}
\dot{x}=X(x)+\mathbf{u}(x)
\end{equation}
conserves $F_1,...,F_k$ and dissipates $G$ with a given rate. These vector fields are the solutions of \eqref{problema}.
The next result gives the dissipation behavior of the function $G$ along the solutions of the perturbed system \eqref{perturbat}, where $\mathbf{u}=\mathbf{v_0}+\mathbf{w}$.

\begin{thm}\label{crescator}
The function $G$ increases along the solutions of the system
$$\dot{x}=X(x)+\mathbf{v_0}(x)+\mathbf{w}(x),$$
where $\mathbf{w}\in \mathcal{X}(M)$ with $\mathbf{w}(x)\in Sp^{\perp}[\nabla
F_1(x),...,\nabla F_k(x),\nabla G(x)]$.
\end{thm}

\begin{proof}
The standard control vector field $\mathbf{v}_0$ is a solution of \eqref{problema} with $h(x)=\det\Sigma_{(F_1,...,F_k,G)}^{(F_1,...,F_k,G)}(x)\geq 0$ for any $x\in M$. Consequently,
$$\frac{dG}{dt}(x(t,x_0))=h(x(t,x_0))=\det\Sigma_{(F_1,...,F_k,G)}^{(F_1,...,F_k,G)}(x(t,x_0))\geq 0,$$
where $x(\cdot , x_0)$ is the solution of the dynamical system $\dot{x}=\mathbf{v_0}(x)+\mathbf{w}(x)$ with the initial condition $x(0,x_0)=x_0$.
\end{proof}

\section{The covariant formulation of the standard control vector field}

In the Euclidean case, a dissipation was constructed in \cite{indian} that preserves $k$ conserved quantities of a dynamical system and dissipates another conserved quantity. We generalize this construction to a general Riemannian manifold and moreover, we prove that this dissipation is minus the standard control vector field $\mathbf{v}_0$.

In analogy to \cite{indian}, we introduce the following one form
\begin{equation*}\label{one-form}
    \omega=(-1)^n*(dF_1\wedge ... \wedge dF_k\wedge *(dG\wedge dF_1\wedge ... \wedge dF_k)),
\end{equation*}
where $*$ is the Hodge star operator associated to the Riemannian metric $g$.
The dissipation vector field $\sharp_g(\omega)$, for the Euclidean case, was introduced in \cite{indian}. Next we will prove that this dissipation is precisely the standard control vector field, i.e. $\mathbf{v}_0=-\sharp_g(\omega)$, where $\sharp_g:\Omega^1(M)\rightarrow \mathcal{X}(M)$ is associated with the Riemannian metric $g$.

In local coordinates we have
$$dF_1=\frac{\partial F_1}{\partial x^{a_1}}dx^{a_1},...,dF_k=\frac{\partial F_k}{\partial x^{a_k}}dx^{a_k}, dG=\frac{\partial G}{\partial x^{a_{k+1}}}dx^{a_{k+1}}.$$
By direct computation  we obtain
$$*(dF_1\wedge ... \wedge dF_k\wedge dG)\shortstack[pos]{\small{\eqref{Hodge-formula}}\\\,=\,}\frac{\sqrt{|g|}}{(n-k-1)!}\frac{\partial F_1}{\partial x^{a_1}}...\frac{\partial F_k}{\partial x^{a_k}}\frac{\partial G}{\partial x^{a_{k+1}}}g^{a_1l_1}...g^{a_{k+1}l_{k+1}}\epsilon_{l_1...l_{k+1}l_{k+2}...l_n}dx^{l_{k+2}}\wedge...\wedge dx^{l_n}.$$
Consequently, we have the following computation
\begin{eqnarray*}
\omega&=&(-1)^{n+k}*(dF_1\wedge ... \wedge dF_k\wedge *(dF_1\wedge ... \wedge dF_k\wedge dG))\nonumber \\
&\shortstack[pos]{\small{\eqref{Hodge-formula}}\\\,\nonumber \\
=\,}&\frac{(-1)^{n+k}|g|}{(n-k-1)!}\frac{\partial F_1}{\partial x^{b_1}}...\frac{\partial F_k}{\partial x^{b_k}}\frac{\partial F_1}{\partial x^{a_1}}...\frac{\partial F_k}{\partial x^{a_k}}\frac{\partial G}{\partial x^{a_{k+1}}}\nonumber \\
& &g^{a_1l_1}...g^{a_{k+1}l_{k+1}}g^{p_{k+2}l_{k+2}}...g^{p_nl_n}\epsilon_{l_1...l_{k+1}l_{k+2}...l_n}
g^{b_1s_1}...g^{b_ks_k}\epsilon_{s_1...s_kp_{k+2}...p_n q}dx^q \nonumber \\
&\shortstack[pos]{\small{\eqref{determinant-amestecat}}\\\,=\,}&\frac{(-1)^{n+k}|g|}{(n-k-1)!}\frac{\partial F_1}{\partial x^{b_1}}...\frac{\partial F_k}{\partial x^{b_k}}\frac{\partial F_1}{\partial x^{a_1}}...\frac{\partial F_k}{\partial x^{a_k}}\frac{\partial G}{\partial x^{a_{k+1}}}\nonumber \\
& &|g^{-1}|\epsilon_{a_1...a_ka_{k+1}p_{k+2}...p_n} g^{b_1s_1}...g^{b_ks_k}\epsilon_{s_1...s_kp_{k+2}...p_n q}dx^q \nonumber \\
&\shortstack[pos]{\small{\eqref{q-sarit}}\\\,=\,}&\frac{(-1)^{n+k}}{(n-k-1)!}\frac{\partial F_1}{\partial x^{b_1}}...\frac{\partial F_k}{\partial x^{b_k}}\frac{\partial F_1}{\partial x^{a_1}}...\frac{\partial F_k}{\partial x^{a_k}}\frac{\partial G}{\partial x^{a_{k+1}}}\nonumber \\
& &\epsilon_{a_1...a_ka_{k+1}p_{k+2}...p_n} g^{b_1s_1}...g^{b_ks_k}(-1)^{n-k-1}\epsilon_{s_1...s_kqp_{k+2}...p_n }dx^q \nonumber \\
&\shortstack[pos]{\small{\eqref{Ricci-Kroneker}+\eqref{q-sumare-factorial}}\\\,=\,}&-\frac{1}{(n-k-1)!}\frac{\partial F_1}{\partial x^{b_1}}...\frac{\partial F_k}{\partial x^{b_k}}\frac{\partial F_1}{\partial x^{a_1}}...\frac{\partial F_k}{\partial x^{a_k}}\frac{\partial G}{\partial x^{a_{k+1}}} g^{b_1s_1}...g^{b_ks_k}(n-k-1)!\delta_{a_1...a_ka_{k+1}}^{s_1...s_kq}dx^q \nonumber \\
&=&-(\sum_{i=1}^{k+1}\omega_i),
\end{eqnarray*}
where $\omega_i=\frac{\partial F_1}{\partial x^{b_1}}...\frac{\partial F_k}{\partial x^{b_k}}\frac{\partial F_1}{\partial x^{a_1}}...\frac{\partial F_k}{\partial x^{a_k}}\frac{\partial G}{\partial x^{a_{k+1}}} g^{b_1s_1}...g^{b_ks_k}\delta_{a_1...a_ka_{k+1}}^{s_1...s_ka_i}dx^{a_i}$.
\medskip

\noindent We need to analyze the one forms $\omega_1,...,\omega_k,\omega_{k+1}$. We have
\begin{eqnarray*}
\omega_{k+1}&=&\frac{\partial F_1}{\partial x^{b_1}}...\frac{\partial F_k}{\partial x^{b_k}}\frac{\partial F_1}{\partial x^{a_1}}...\frac{\partial F_k}{\partial x^{a_k}}\frac{\partial G}{\partial x^{a_{k+1}}} g^{b_1s_1}...g^{b_ks_k}\delta_{a_1...a_ka_{k+1}}^{s_1...s_ka_{k+1}}dx^{a_{k+1}}\nonumber \\
&\shortstack[pos]{\small{\eqref{q-constant}}\\\,=\,}&\frac{\partial F_1}{\partial x^{b_1}}...\frac{\partial F_k}{\partial x^{b_k}}\frac{\partial F_1}{\partial x^{a_1}}...\frac{\partial F_k}{\partial x^{a_k}}\delta_{a_1...a_k}^{s_1...s_k}\frac{\partial G}{\partial x^{a_{k+1}}}dx^{a_{k+1}}\nonumber \\
&=&\frac{\partial F_1}{\partial x^{b_1}}...\frac{\partial F_k}{\partial x^{b_k}}\frac{\partial F_1}{\partial x^{\sigma(s_1)}}...\frac{\partial F_k}{\partial x^{\sigma(s_k)}}g^{b_1s_1}...g^{b_ks_k}\delta_{\sigma(s_1)...\sigma(s_k)}^{s_1...s_k}\frac{\partial G}{\partial x^{a_{k+1}}}dx^{a_{k+1}}\nonumber \\
&\shortstack[pos]{\small{\eqref{sigma-developare}}\\\,=\,}&\frac{\partial F_1}{\partial x^{b_1}}...\frac{\partial F_k}{\partial x^{b_k}}g^{b_1s_1}...g^{b_ks_k}(\sum_{\sigma}sgn(\sigma)\frac{\partial F_1}{\partial x^{\sigma(s_1)}}...\frac{\partial F_k}{\partial x^{\sigma(s_k)}})\frac{\partial G}{\partial x^{a_{k+1}}}dx^{a_{k+1}}\nonumber \\
&=&\det\Sigma_{(F_1,...,F_k)}^{(F_1,...,F_k)}dG.
\end{eqnarray*}
Also
\begin{eqnarray*}
\omega_1&\shortstack[pos]{\small{\eqref{q-sarit}}\\\,=\,}&(-1)^k\frac{\partial F_1}{\partial x^{b_1}}...\frac{\partial F_k}{\partial x^{b_k}}\frac{\partial F_2}{\partial x^{a_2}}...\frac{\partial F_k}{\partial x^{a_k}}\frac{\partial G}{\partial x^{a_{k+1}}} g^{b_1s_1}...g^{b_ks_k}\delta_{a_2...a_{k+1}a_1}^{s_1...s_ka_1}\frac{\partial F_1}{\partial x^{a_1}}dx^{a_1}\nonumber \\
&=&(-1)^k\frac{\partial F_1}{\partial x^{b_1}}...\frac{\partial F_k}{\partial x^{b_k}}g^{b_1s_1}...g^{b_ks_k}(\sum_{\sigma}sgn(\sigma)\frac{\partial G}{\partial x^{\sigma(s_k)}}\frac{\partial F_2}{\partial x^{\sigma(s_1)}}...\frac{\partial F_k}{\partial x^{\sigma(s_{k-1})}})dF_1\nonumber \\
&\shortstack[pos]{\small{\eqref{sigma-developare}}\\\,=\,}&(-1)^k\det\Sigma_{(\hat{F_1},F_2,...,F_k,G)}^{(F_1,...,F_k)}dF_1.
\end{eqnarray*}
Analogously, for $i=\overline{2,k}$,
$\omega_i=(-1)^{k+i+1}\det\Sigma_{(F_1,...,\hat{F_i},...,F_k,G)}^{(F_1,...,F_k)}dF_i.$
\medskip

\noindent As a consequence of the above computations and of the definition of the standard control vector field \eqref{v0}, we obtain the following result.

\begin{thm}
On the Riemannian manifold $(M,g)$, we have the equivalent description of the standard control vector field
$$\mathbf{v}_0=(-1)^{n+1}\sharp_g(*(dF_1\wedge ... \wedge dF_k\wedge *(dG\wedge dF_1\wedge ... \wedge dF_k))).$$
\end{thm}

\section{The gradient formulation of the standard control vector field}

The standard control vector field $\mathbf{v_0}$ is tangent to every regular leaf $L_c:=\mathbf{F}^{-1}(c)$ generated by the regular values of the function $\mathbf{F}:=(F_1,...,F_k):M\rightarrow \mathbb{R}^k$.
We will endow every regular leaf $L_c$ with a Riemannian metric $\tau_c$ such that $\mathbf{v_0}$ is a gradient vector field on $L_c$, i.e.
$$\mathbf{v_0}_{_{|L_c}}=\nabla_{\tau_c}G_{|L_c}.$$

In order to do this we first construct a degenerate symmetric contravariant 2-tensor $\mathbf{T}$ on the manifold $M$ that is nondegenerate when restricted to every regular leaf $L_c$. The Riemannian metric $\tau_c$ will be
$\tau_c:=\mathbf{T}_{|L_c}^{-1}$.

In what follows, we will construct the tensor $\mathbf{T}$. In Riemannian geometry we can write the gradient vector field of the function $G$ as
$\nabla G=\mathbf{i}_{dG}g^{-1},$
where $g^{-1}$ is the contravariant 2-tensor $g^{-1}(x)=g^{pq}(x)\frac{\partial}{\partial x^p}\otimes \frac{\partial}{\partial x^q}$ constructed from the metric tensor $g$ and $\mathbf{i}$ is the interior product. We recall the following standard results in Riemannian geometry which will be used several times in this section,
$$\nabla H(\alpha)=g^{-1}(dH,\alpha)=\alpha (\nabla H),\,\,\,dH(\nabla K)=<\nabla H, \nabla K>,$$
where $\alpha\in \Omega^1(M)$ and $H,K\in \mathcal{C}^{\infty}(M)$.

We have the following contravariant 2-tensor
$\nabla F_i\otimes\nabla F_j :\Omega^1(M)\times\Omega^1(M)\rightarrow \mathbb{R}$
$$\nabla F_i\otimes\nabla F_j\,\,(\alpha,\beta)=\alpha (\nabla F_i)\beta(\nabla F_j).$$

\begin{lem}\label{ajutatoare}
For $i,j\in\{1,...,k\}$ we have the equalities
\begin{itemize}
  \item [(i)] $\mathbf{i}_{dG}(\nabla F_i\otimes\nabla F_j)=<\nabla G,\nabla F_i>\nabla F_j$
  \item [(ii)] $\nabla F_i\otimes\nabla F_j=g^{rp}g^{sq}\frac{\partial F_i}{\partial x^r}\frac{\partial F_j}{\partial x^s}\frac{\partial}{\partial x^p}\otimes \frac{\partial}{\partial x^q}.$
\end{itemize}
\end{lem}

\begin{proof}
{\it (i)} The proof is a direct computation,
$$\mathbf{i}_{dG}(\nabla F_i\otimes\nabla F_j)(\alpha)=\nabla F_i\otimes\nabla F_j\,\,(dG,\alpha)=dG(\nabla F_i)\alpha(\nabla F_j)=<\nabla G,\nabla F_i>\nabla F_j(\alpha),$$
for any $\alpha \in \Omega^1(M)$.

\noindent {\it (ii)} By straightforward computation we have
$$\nabla F_i\otimes\nabla F_j\,\,(dx^p,dx^q)=dx^p(\nabla F_i)dx^q(\nabla F_j)=g^{-1}(dF_i,dx^p)g^{-1}(dF_j,dx^q)=g^{rp}g^{sq}\frac{\partial F_i}{\partial x^r}\frac{\partial F_j}{\partial x^s}.$$

\end{proof}

We define the symmetric contravariant 2-tensor $\mathbf{T}:\Omega^1(M)\times\Omega^1(M)\rightarrow \mathbb{R}$ by
\begin{equation}\label{T}
\mathbf{T}:=\sum_{i,j=1}^k(-1)^{i+j+1}\det\Sigma_{(F_1,...,\hat{F_i},...,F_k)}^{(F_1,...,\hat{F_j},...,F_k)}\nabla F_i\otimes\nabla F_j+\det\Sigma_{(F_1,...,F_k)}^{(F_1,...,F_k)} g^{-1}.
\end{equation}
For proving the symmetry of $\mathbf{T}$ we notice that $\Sigma_{(F_1,...,\hat{F_i},...,F_k)}^{(F_1,...,\hat{F_j},...,F_k)}=(\Sigma_{(F_1,...,\hat{F_j},...,F_k)}^{(F_1,...,\hat{F_i},...,F_k)})^T$ and consequently, $\det\Sigma_{(F_1,...,\hat{F_i},...,F_k)}^{(F_1,...,\hat{F_j},...,F_k)}=\det\Sigma_{(F_1,...,\hat{F_j},...,F_k)}^{(F_1,...,\hat{F_i},...,F_k)}$.
The contravariance can be deduced by the expressions in local coordinates, namely the coefficients of the symmetric 2-tensor $\mathbf{T}$ are
\begin{eqnarray*}
\mathbf{T}^{pq}&=&\sum_{i,j=1}^k(-1)^{i+j}g^{a_1b_1}...\widehat{g^{a_jb_j}}...g^{a_kb_k}g^{rp}g^{sq}\frac{\partial F_1}{\partial x^{b_1}}...\widehat{\frac{\partial F_j}{\partial x^{b_j}}}...\frac{\partial F_k}{\partial x^{b_k}}\frac{\partial F_i}{\partial x^r}\frac{\partial F_j}{\partial x^s}\eta_{a_1...\widehat{a_j}...a_k}(F_1...\widehat{F_i}...F_k) \nonumber \\
& &-g^{a_1b_1}...g^{a_kb_k}g^{pq}\frac{\partial F_1}{\partial x^{b_1}}...\frac{\partial F_k}{\partial x^{b_k}}\eta_{a_1...a_k}(F_1...F_k),
\end{eqnarray*}
where we have used \eqref{sigma-developare} and made the notation $\eta_{a_1...a_k}(F_1...F_k):=\det\left(%
\begin{array}{ccccccc}
    \frac{\partial F_1}{\partial x^{a_1}}& ... &  \frac{\partial F_k}{\partial x^{a_1}} \\
    ... & ...  & ... \\
   \frac{\partial F_1}{\partial x^{a_k}} & ... &  \frac{\partial F_k}{\partial x^{a_k}} \\
\end{array}%
\right).$
\medskip

\begin{rem}\label{T-expresie1F}
For the case $k=1$ and using the notation $F_1=F$, the expression $\det\Sigma_{(F_1,...,\hat{F_i},...,F_k)}^{(F_1,...,\hat{F_j},...,F_k)}$ becomes the constant function $1$ and
the expression $\det\Sigma_{(F_1,...,F_k)}^{(F_1,...,F_k)}$ becomes $||\nabla F||^2$, where $||\cdot||$ is the norm generated by the Riemannian metric $g$.
Consequently, the symmetric contravariant 2-tensor $\mathbf{T}$ has the form
\begin{equation}\label{T-cu1F}
    \mathbf{T}=-\nabla F\otimes\nabla F+||\nabla F||^2g^{-1}.
\end{equation}
In local coordinates we have the expresion
\begin{equation*}
\mathbf{T}^{pq}=(g^{ap}g^{bq}-g^{ab}g^{pq})\frac{\partial F}{\partial x^a}\frac{\partial F}{\partial x^b}.
\end{equation*}
For the Euclidean case, the above expression of the tensor $\mathbf{T}$ with $F$ being the Hamiltonian function of a Hamilton-Poisson system was used in \cite{birtea-comanescu}.
\begin{flushright}
$\bigtriangleup$
\end{flushright}
\end{rem}

Next, we study a few properties of the symmetric contravariant 2-tensor $\mathbf{T}$ that we need in what follows.
The functions $F_1,...,F_k:M\rightarrow \mathbb{R}$ generate the following distribution on $M$, $$\mathcal{X}_{tan}(M)=\{X\in\mathcal{X}(M)\,|\,dF_s(X)=0,\,s=\overline{1,k}\}$$
and its dual distribution
$$\Omega^1_{tan}(M)=\{\alpha\in \Omega^1(M)\,|\,\alpha(\nabla F_s)=0,\,s=\overline{1,k}\}.$$

\begin{prop}\label{proprietatiT}
We have the following results:
\begin{itemize}
  \item [(i)] $\nabla F_i\otimes \nabla F_j(\alpha, \beta)=0,\,\,\forall \alpha \in \Omega^1_{tan}(M),\forall \beta \in \Omega^1(M)$.
\item [(ii)] $\mathbf{T}(\alpha,\beta)=\det\Sigma_{(F_1,...,F_k)}^{(F_1,...,F_k)} g^{-1}(\alpha,\beta),\,\,\forall \alpha \in \Omega^1_{tan}(M),\forall \beta \in \Omega^1(M)$.
\item [(iii)] $\mathbf{T}(\alpha,\alpha)=\det\Sigma_{(F_1,...,F_k)}^{(F_1,...,F_k)} ||\alpha||^2,\,\,\forall \alpha \in \Omega^1_{tan}(M)$.
\item [(iv)] $\mathbf{T}(dF_s,\beta) =0,\,\,\forall s=\overline{1,k},\,\,\forall \beta\in \Omega^1(M)$.
\end{itemize}
\end{prop}

\begin{proof}
$(i)$ By definition we have  $\nabla F_i\otimes \nabla F_j(\alpha, \beta)=\alpha(\nabla F_i)\beta(\nabla F_j)$. Because $\alpha\in \Omega^1_{tan}(M)$ we have the equality $\alpha(\nabla F_i)=0$.

$(ii)$ From the definition of $\mathbf{T}$ we have that for any $\alpha \in \Omega^1_{tan}(M)$ and for any $\beta \in \Omega^1(M)$,
\begin{eqnarray*}
\mathbf{T}(\alpha,\beta)&=&\sum_{i,j=1}^k(-1)^{i+j+1}\det\Sigma_{(F_1,...,\hat{F_i},...,F_k)}^{(F_1,...,\hat{F_j},...,F_k)}\nabla F_i\otimes\nabla F_j(\alpha,\beta)+\det\Sigma_{(F_1,...,F_k)}^{(F_1,...,F_k)} g^{-1}(\alpha,\beta)\nonumber \\
&=&\det\Sigma_{(F_1,...,F_k)}^{(F_1,...,F_k)} g^{-1}(\alpha,\beta).
\end{eqnarray*}

$(iii)$ From $(ii)$ we have that for any $\alpha \in \Omega^1_{tan}(M)$
$$\mathbf{T}(\alpha,\alpha)=\det\Sigma_{(F_1,...,F_k)}^{(F_1,...,F_k)} g^{-1}(\alpha,\alpha)=\det\Sigma_{(F_1,...,F_k)}^{(F_1,...,F_k)} ||\alpha||^2.$$

$(iv)$ From Lemma \eqref{ajutatoare} (i), for any $s\in \overline{1,k}$ we obtain
\begin{eqnarray*}
\mathbf{i}_{dF_s}\mathbf{T}&=&\sum_{i,j=1}^k(-1)^{i+j+1}\det\Sigma_{(F_1,...,\hat{F_i},...,F_k)}^
{(F_1,...,\hat{F_j},...,F_k)}\mathbf{i}_{dF_s}\nabla F_i\otimes\nabla F_j+\det\Sigma_{(F_1,...,F_k)}^{(F_1,...,F_k)} \mathbf{i}_{dF_s}g^{-1}\nonumber \\
&=&\sum_{i,j=1}^k(-1)^{i+j+1}\det\Sigma_{(F_1,...,\hat{F_i},...,F_k)}^
{(F_1,...,\hat{F_j},...,F_k)}<\nabla F_s,\nabla F_i>\nabla F_j+\det\Sigma_{(F_1,...,F_k)}^{(F_1,...,F_k)}\nabla F_s\nonumber \\
&=&\sum_{i=1}^k(-1)^{k+i+1}<\nabla F_s,\nabla F_i>(\sum_{j=1}^k(-1)^{k+j}\det\Sigma_{(F_1,...,\hat{F_i},...,F_k)}^
{(F_1,...,\hat{F_j},...,F_k)}\nabla F_j)+\det\Sigma_{(F_1,...,F_k)}^{(F_1,...,F_k)}\nabla F_s \\
&=&\sum_{i=1}^k(-1)^{k+i+1}<\nabla F_s,\nabla F_i>\det\left(%
\begin{array}{ccccccc}
   <\nabla F_{1},\nabla F_{1}> & ... & <\nabla F_k,\nabla F_{1}> \\
   ... & ... & ... \\
   \widehat{<\nabla F_{1},\nabla F_i>} & ... & \widehat{<\nabla F_k,\nabla F_i>} \\
    ... &  ...  & ... \\
 <\nabla F_{1},\nabla F_k> & ... & <\nabla F_k,\nabla F_k> \\
    \nabla F_{1}& ... & \nabla F_k \\
\end{array}%
\right) \\
& &+\det\Sigma_{(F_1,...,F_k)}^{(F_1,...,F_k)}\nabla F_s \\
&=&\det\left(%
\begin{array}{ccccccc}
   <\nabla F_{1},\nabla F_{1}> & ... & <\nabla F_k,\nabla F_{1}>& <\nabla F_s,\nabla F_{1} \\
   ... & ... & ... & ... \\
 <\nabla F_{1},\nabla F_k> & ... & <\nabla F_k,\nabla F_k> & <\nabla F_s,\nabla F_k> \\
    \nabla F_{1}& ... & \nabla F_k & \nabla F_s \\
\end{array}%
\right)=0.
\end{eqnarray*}
\end{proof}

Property $(iv)$ of the above proposition shows that the 2-tensor $\mathbf{T}$ is degenerate and consequently it is not the inverse of any covariant metric 2-tensor.
 Nevertheless, the standard control vector field $\mathbf{v}_0$ still behaves like a gradient vector field with respect to the degenerate symmetric contravariant 2-tensor $\mathbf{T}$.

\begin{thm}\label{v0-cu-i}
On the manifold $(M,g)$, the standard control vector field $\mathbf{v_0}$ is given by the following formula,
$$\mathbf{v_0}=\mathbf{i}_{dG}(\mathbf{T}).$$
\end{thm}

\begin{proof}
By developing the determinants $\det\Sigma_{(F_1,...,\hat{F_j},...,F_k,G)}^
{(F_1,...,F_k)}$ after the last column in the expression \eqref{v0}, we obtain
\begin{eqnarray*}
\mathbf{v_0}&=&\sum_{j=1}^k(-1)^{j+k+1}\det\Sigma_{(F_1,...,\hat{F_j},...,F_k,G)}^
{(F_1,...,F_k)}\nabla F_i+\det\Sigma_{(F_1,...,F_k)}^{(F_1,...,F_k)}\nabla G \\
&=&\sum_{j=1}^k(-1)^{j+k+1}(\sum_{i=1}^k(-1)^{i+k}\det\Sigma_{(F_1,...,\hat{F_j},...,F_k)}^
{(F_1,...,\hat{F_i},...F_k)}<\nabla G,\nabla F_i>)\nabla F_j+\det\Sigma_{(F_1,...,F_k)}^{(F_1,...,F_k)}\nabla G \\
&=&\sum_{i,j=1}^k(-1)^{i+j+1}\det\Sigma_{(F_1,...,\hat{F_i},...,F_k)}^
{(F_1,...,\hat{F_j},...F_k)}<\nabla G,\nabla F_i>\nabla F_j+\det\Sigma_{(F_1,...,F_k)}^{(F_1,...,F_k)}\nabla G \\
&=&\mathbf{i}_{dG}(\mathbf{T}).
\end{eqnarray*}
In the above proof we have used the fact that $\Sigma_{(F_1,...,\hat{F_j},...,F_k)}^
{(F_1,...,\hat{F_i},...F_k)}=(\Sigma_{(F_1,...,\hat{F_i},...,F_k)}^
{(F_1,...,\hat{F_j},...F_k)})^T$.
\end{proof}

The symmetric contravariant 2-tensor $\mathbf{T}$ is degenerate and consequently it cannot be inverted. We will prove that it is invertible when restricted to $\Omega^1_{tan}(M)$. Throughout the remaining of this section we will assume that $\det\Sigma_{(F_1,...,F_k)}^{(F_1,...,F_k)}\neq 0$ on the whole manifold $M$. If this assumption is not true, then we replace $M$ with the open subset of regular points, i.e. $\{x\in M\,|\,\det\Sigma_{(F_1,...,F_k)}^{(F_1,...,F_k)}(x)\neq 0\}$.

Having a Riemannian manifold $(M,g)$, we recall the following well known notions. One can define the following operators $\flat_g:\mathcal{X}(M)\rightarrow \Omega^1(M)$, $\flat_g(X):=g(X,\cdot)$ and $\sharp_g:\Omega^1(M)\rightarrow \mathcal{X}(M)$, $\sharp_g(\alpha):=g^{-1}(\alpha,\cdot)$. The nondegeneracy of the metric tensor $g$ implies $\flat_g=\sharp_g^{-1}$.
For the tensor $\mathbf{T}$, we can define the operator $$\sharp_T:\Omega^1(M)\rightarrow \mathcal{X}(M),\,\,\sharp_T(\alpha):=\mathbf{T}(\alpha,\cdot).$$

We will prove by double inclusion the set equality $\sharp_T(\Omega^1_{tan}(M))=\mathcal{X}_{tan}(M)$. Indeed, for $\alpha\in \Omega^1_{tan}(M)$ and using Proposition \ref{proprietatiT} (iv), we have $dF_s(\sharp_T(\alpha))=\sharp_T(\alpha)(dF_s)=\mathbf{T}(\alpha,dF_s)=0$ which implies that $\sharp_T(\alpha)\in \mathcal{X}_{tan}(M)$. For the other inclusion, let $X_0\in \mathcal{X}_{tan}(M)$ and $\alpha_0:=\frac{1}{\det\Sigma_{(F_1,...,F_k)}^{(F_1,...,F_k)}}\flat_g(X_0)$. We will prove that $\sharp_T(\alpha_0)=X_0$. For this, we first need to show that $\alpha_0$ introduced above is an element in $\Omega^1_{tan}(M)$. We have $\alpha_0(\nabla F_s)=\frac{1}{\det\Sigma_{(F_1,...,F_k)}^{(F_1,...,F_k)}}\flat_g(X_0)(\nabla F_s)=\frac{1}{\det\Sigma_{(F_1,...,F_k)}^{(F_1,...,F_k)}}g(X_0,\nabla F_s)=\frac{1}{\det\Sigma_{(F_1,...,F_k)}^{(F_1,...,F_k)}}dF_s(X_0)$. Because $X_0\in \mathcal{X}_{tan}(M)$, we obtain that $\alpha_0(\nabla F_s)=0,\,\,\forall s=\overline{1,k}$, which implies that $\alpha_0\in \Omega^1_{tan}(M)$.
By direct computation we have
$$\sharp_T(\alpha_0)(\beta)= \mathbf{T}(\alpha_0,\beta)\shortstack[pos]{\small{P.\ref{proprietatiT}\,(ii)}\\\,=\,}\det\Sigma_{(F_1,...,F_k)}^{(F_1,...,F_k)}
g^{-1}(\alpha_0,\beta)=g^{-1}(\flat_g(X_0),\beta)=X_0(\beta),$$
for any $\beta\in \Omega^1(M)$.

The operator $\sharp_T$ is injective on $\Omega^1_{tan}(M)$. Indeed, for $\alpha_1, \alpha_2 \in \Omega^1_{tan}(M)$, suppose that $\sharp_T(\alpha_1)=\sharp_T(\alpha_2)$. This is equivalent with $\mathbf{T}(\alpha_1,\beta)=\mathbf{T}(\alpha_2,\beta)$, for all $\beta\in\Omega^1(M)$. Using Proposition \ref{proprietatiT} (ii), we obtain the equality $\det\Sigma_{(F_1,...,F_k)}^{(F_1,...,F_k)}g^{-1}(\alpha_1,\beta)=
\det\Sigma_{(F_1,...,F_k)}^{(F_1,...,F_k)}g^{-1}(\alpha_2,\beta)$, for all $\beta\in\Omega^1(M)$. By the nondegeneracy of the metric tensor $g$ we obtain that $\alpha_1=\alpha_2$.

The restricted operator $\sharp_T:\Omega^1_{tan}(M)\rightarrow \mathcal{X}_{tan}(M)$ is invertible and consequently we can define the inverse operator $\flat_T:\mathcal{X}_{tan}(M)\rightarrow \Omega^1_{tan}(M)$. From the above considerations we obtain the equality
\begin{equation}\label{flatT}
\flat_T(X)=\frac{1}{\det\Sigma_{(F_1,...,F_k)}^{(F_1,...,F_k)}}\flat_g(X),
\end{equation}
for all $X\in \mathcal{X}_{tan}(M)$.

\begin{lem}\label{BT}
For $\alpha\in \Omega^1(M)$ and $X_0\in \mathcal{X}_{tan}(M)$ we have the equality
$$\mathbf{T}(\alpha,\flat_T(X_0))=\alpha(X_0).$$
\end{lem}

\begin{proof}
By direct computation we have
$$\mathbf{T}(\alpha,\flat_T(X_0))\shortstack[pos]{\small{P.\ref{proprietatiT}\,(ii)}\\\,=\,}\det\Sigma_{(F_1,...,F_k)}^{(F_1,...,F_k)}
g^{-1}(\alpha,\flat_T(X_0))=g^{-1}(\flat_T(X_0),\alpha)=\alpha(X_0).$$
\end{proof}

\begin{defn}
We introduce the symmetric nondegenerate covariant 2-tensor $\mathbf{T}^{-1}:\mathcal{X}_{tan}(M)\times \mathcal{X}_{tan}(M)\rightarrow \mathcal{C}^{\infty}(M)$ defined by
\begin{equation}\label{inversT}
\mathbf{T}^{-1}(X,Y)=\mathbf{T}(\flat_T(X),\flat_T(Y)).
\end{equation}
\end{defn}

On every regular leaf $L_c=\mathbf{F}^{-1}(c)$ we will construct a Riemannian metric $\tau_c$ using the tensor $\mathbf{T}^{-1}$. Let $i_c:L_c\rightarrow M$ be the canonical inclusion of the regular leaf $L_c$ into the manifold $M$. We have the following inclusion $i_{c_*}(\mathcal{X}(L_c))\subset \mathcal{X}_{tan}(M)$.
\begin{defn}
On a regular leaf $L_c$ we define the Riemannian metric $\tau_c:\mathcal{X}(L_c)\times \mathcal{X}(L_c)\rightarrow \mathcal{C}^{\infty}(L_c)$
$$\tau_c(X^c,Y^c):=\mathbf{T}^{-1}(i_{c_*}X^c,i_{c_*}Y^c).$$
\end{defn}

The next result gives the formula for the standard control vector field by using coordinates on the regular leaf $L_c$. More precisely, we prove that $\mathbf{v}_0$ restricted to $L_c$ is a gradient vector field with respect to the Riemannian metric $\tau_c$. Moreover, we prove that this metric is a conformal metric with respect to the first fundamental form induced by the ambient metric $g$ on the submanifold $L_c$.

\begin{thm}\label{caracterizare}
On a regular leaf $L_c$ we have the following characterizations.
\begin{itemize}
\item [(i)] $\tau_c=\frac{1}{\det\Sigma_{(F_1,...,F_k)}^{(F_1,...,F_k)}\circ i_c}i_c^*g$
\item [(ii)] The standard control vector field $\mathbf{v_0}$ is a vector field in $\mathcal{X}_{tan}(M)$. Moreover, there exists a vector field $\mathbf{v}_0^c\in \mathcal{X}(L_c)$ such that $i_{c_*}(\mathbf{v}_0^c)=\mathbf{v}_{0 |L_c}$, where $\mathbf{v}_{0 |L_c}$ is the restriction of $\mathbf{v}_0$ to the submanifold $i_c(L_c)$.
\item [(iii)] $\mathbf{v_0}_{|L_c}=i_{c_*}\nabla_{\tau_c}(G\circ i_c)$.
\end{itemize}
\end{thm}

\begin{proof}
{\it (i)} We have the following computations
\begin{eqnarray*}
\tau_c(X^c,Y^c)&=&\mathbf{T}^{-1}(i_{c_*}X^c,i_{c_*}Y^c)\shortstack[pos]{\small{\eqref{inversT}}\\\,=\,}
\mathbf{T}(\flat_T(i_{c_*}X^c),\flat_T(i_{c_*}Y^c)) \\
&\shortstack[pos]{\small{P.\ref{proprietatiT}\,(ii)}\\\,=\,}&
\det\Sigma_{(F_1,...,F_k)}^{(F_1,...,F_k)}g^{-1}(\flat_T(i_{c_*}X^c),\flat_T(i_{c_*}Y^c)) \\
&\shortstack[pos]{\small{\eqref{flatT}}\\\,=\,}&\det\Sigma_{(F_1,...,F_k)}^{(F_1,...,F_k)}
g^{-1}(\frac{1}{\det\Sigma_{(F_1,...,F_k)}^{(F_1,...,F_k)}}\flat_g(i_{c_*}X^c),
\frac{1}{\det\Sigma_{(F_1,...,F_k)}^{(F_1,...,F_k)}}\flat_g(i_{c_*}Y^c)) \\
&=& \frac{1}{\det\Sigma_{(F_1,...,F_k)}^{(F_1,...,F_k)}}g(i_{c_*}X^c,i_{c_*}Y^c)
=\frac{1}{\det\Sigma_{(F_1,...,F_k)}^{(F_1,...,F_k)}\circ i_c}i_c^*g(X^c,Y^c).
\end{eqnarray*}

{\it (ii)} By construction, the standard control vector field $\mathbf{v}_0$ is a solution of \eqref{problema}, which implies that $\mathbf{v}_0\in \mathcal{X}_{tan}(M)$. As $i_c:L_c\rightarrow i_c(L_c)\subset M$ is a diffeomorphism, there exists a vector field $\mathbf{v}_0^c\in \mathcal{X}(L_c)$ such that $i_{c_*}(\mathbf{v}_0^c)=\mathbf{v}_{0 |L_c}$.

{\it (iii)} By the definition of a gradient vector field we have
$$d(G\circ i_c)(Y^c)=\tau_c(\nabla _{\tau_c}(G\circ i_c),Y^c),\,\,\forall Y^c\in \mathcal{X}(L_c).$$
which is equivalent with the equality
$$dG(i_{c_*}Y^c)=\mathbf{T}^{-1}(i_{c_*}\nabla _{\tau_c}(G\circ i_c),i_{c_*}Y^c),\,\,\forall Y^c\in \mathcal{X}(L_c).$$
We have the following computations
$$\mathbf{T}^{-1}(\mathbf{v}_{0 |L_c},i_{c_*}Y^c)=\mathbf{T}(\flat_T(\mathbf{v}_{0 |L_c}),\flat_T(i_{c_*}Y^c))=
g^{-1}(\flat_g(\mathbf{v}_{0 |L_c}),\flat_T(i_{c_*}Y^c))=\mathbf{v}_{0 |L_c}(\flat_T(i_{c_*}Y^c))$$
$$\shortstack[pos]{\small{T.\ref{v0-cu-i}}\\\,=\,}\mathbf{i}_{dG}\mathbf{T}(\flat_T(i_{c_*}Y^c))=\mathbf{T}(dG,\flat_T(i_{c_*}Y^c))
\shortstack[pos]{\small{L.\ref{BT}}\\\,=\,}dG(i_{c_*}Y^c).$$
Consequently we have the equality
$$\mathbf{T}^{-1}(i_{c_*}\nabla _{\tau_c}(G\circ i_c),i_{c_*}Y^c)=\mathbf{T}^{-1}(\mathbf{v}_{0 |L_c},i_{c_*}Y^c),\,\,\forall Y^c\in \mathcal{X}(L_c),$$
or equivalently
$$\mathbf{T}^{-1}(i_{c_*}(\nabla _{\tau_c}(G\circ i_c)-\mathbf{v}_0^c),i_{c_*}Y^c)=0,\,\,\forall Y^c\in \mathcal{X}(L_c).$$
And by the definition of $\tau_c$ we obtain
$$\tau_c(\nabla _{\tau_c}(G\circ i_c)-\mathbf{v}_0^c,Y^c)=0,\,\,\forall Y^c\in \mathcal{X}(L_c).$$
By the nondegeneracy of the Riemannian metric tensor $\tau_c$ we obtain $\nabla _{\tau_c}(G\circ i_c)=\mathbf{v}_0^c$ or equivalently $i_{c_*}\nabla _{\tau_c}(G\circ i_c)=\mathbf{v_0}_{|L_c}$.
\end{proof}

For the case when we have only one conserved quantity $F$, i.e. $k=1$ and $F_1=F$, we obtain that $$\tau_c=\frac{1}{||\nabla F||^2\circ i_c}i_c^*g.$$

\section{The projection method formulation of the standard control vector field}

In this section we prove that the standard control vector field $\mathbf{v_0}$ can be written as a scaled orthogonal projection on the regular leaves $L_c$ of the gradient vector field $\nabla G$. Using this formulation we study the connection between the standard control vector field generated by sets of functionally dependent conserved quantities.

\begin{defn}
Let $W$ be a subspace of a finite dimensional inner product space $(V,<\cdot,\cdot >)$. Denote by $W^{\perp}$ the orthogonal complement of $W$ in $V$. Define $P_W:V\rightarrow V$ by
$$P_W(\mathbf{v})=\mathbf{w},$$
where $\mathbf{v}=\mathbf{w}+\mathbf{u}$ with $\mathbf{w}\in W$ and $\mathbf{u}\in W^{\perp}$.

The linear operator $P_W$ is called the orthogonal projection of $V$ onto $W$ along $W^{\perp}$.
\end{defn}

For $x\in M$ we consider the inner product space $(T_xM,<\cdot,\cdot >)$, where the inner product is generated by the Riemannian metric $g$. For regular points of the function $\mathbf{F}=(F_1,...,F_k):M\rightarrow \mathbb{R}^k$ we consider the subspace $T_xL_c=Ker D\mathbf{F}(x)$. Consequently, we have $T_xL_c^{\perp}=Sp[\nabla F_1(x),...,\nabla F_k(x)]$. Indeed, for $\mathbf{y}\in Ker D\mathbf{F}(x)$ we have $<\mathbf{y},\nabla F_s(x)>=g_{ij}(x)y^ig^{aj}(x)\frac{\partial F_s}{\partial x^a}(x)=\delta_i^ay^i\frac{\partial F_s}{\partial x^a}(x)=y^a\frac{\partial F_s}{\partial x^a}(x)=0$ for any $s=\overline{1,k}$.

We define the following linear operator $P_{T_xL_c}:T_xM\rightarrow T_xM$
\begin{equation}\label{proiector}
P_{T_xL_c}(\mathbf{v})=\frac{1}{\det\Sigma_{(F_1,...,F_k)}^{(F_1,...,F_k)}(x)}\det
\left(%
\begin{array}{cccc}
  <\nabla F_1(x),\nabla F_{1}(x)> & ... & <\nabla F_k(x),\nabla F_{1}(x)> & <\mathbf{v},\nabla F_{1}(x)> \\
  ... & ... & ... & ... \\
  <\nabla F_1(x),\nabla F_{k}(x)> & ... & <\nabla F_k(x),\nabla F_{k}(x)> & <\mathbf{v},\nabla F_{k}(x)>\\
\nabla F_1(x) & ... & \nabla F_k(x) & \mathbf{v} \\
\end{array}%
\right),
\end{equation}
where $\mathbf{v} \in T_xM$.
The operator $P_{T_xL_c}$ is the orthogonal projection of $T_xM$ onto $T_xL_c$ along $T_xL_c^{\perp}$. Indeed, we observe that $P_{T_xL_c}(\nabla F_i(x))=0$, for all $i=\overline{1,k}$, as two columns in the determinant become equal. Consequently, $P_{T_xL_c}(\mathbf{u})=0$ for any $\mathbf{u}\in T_xL_c^{\perp}$. For $\mathbf{w}\in T_xL_c$ we have
$$P_{T_xL_c}(\mathbf{w})=\frac{1}{\det\Sigma_{(F_1,...,F_k)}^{(F_1,...,F_k)}(x)}\det
\left(%
\begin{array}{cccc}
  <\nabla F_1(x),\nabla F_{1}(x)> & ... & <\nabla F_k(x),\nabla F_{1}(x)> & 0 \\
  ... & ... & ... & ... \\
  <\nabla F_1(x),\nabla F_{k}(x)> & ... & <\nabla F_k(x),\nabla F_{k}(x)> & 0\\
\nabla F_1(x) & ... & \nabla F_k(x) & \mathbf{w} \\
\end{array}%
\right)=\mathbf{w}.$$
\medskip
Also, the standard control vector field $\mathbf{v}_0$, defined by the equation \eqref{v0}, can be formally written as
\begin{equation}
\mathbf{v}_0(x)=\det \left(%
\begin{array}{cccc}
  <\nabla F_1(x),\nabla F_{1}(x)> & ... & <\nabla F_k(x),\nabla F_{1}(x)> & <\nabla G(x),\nabla F_1(x)> \\
  ... & ... & ... & ... \\
  <\nabla F_1(x),\nabla F_{k}(x)> & ... & <\nabla F_k(x),\nabla F_{k}(x)> & <\nabla G(x),\nabla F_k(x)>\\
\nabla F_1(x) & ... & \nabla F_k(x) & \nabla G(x) \\
\end{array}%
\right)
\end{equation}
From the above considerations we can conclude the following result.

\begin{thm}\label{v0-proiector}
The standard control vector field can be written as a scaled orthogonal projection on the regular leaves $L_c$ of the gradient vector field $\nabla G$, i.e. for $x\in L_c$, we have
$$\mathbf{v}_0(x)=\det\Sigma_{(F_1,...,F_k)}^{(F_1,...,F_k)}(x) P_{T_xL_c} (\nabla G(x)).$$
\end{thm}
\medskip

In what follows we study the connection between the standard control vector fields generated by two sets of functionally dependent conservation laws. More precisely, let
\begin{equation*}\label{}
    \left\{
      \begin{array}{ll}
        H_1=h_1(F_1,...,F_k) \\
        ... \\
         H_k=h_k(F_1,...,F_k)
      \end{array},
    \right.
\end{equation*}
where $h=(h_1,...,h_k):\mathbb{R}^k\rightarrow \mathbb{R}^k$ is a local diffeomorphism. By a straightforward computation we have the following equality, see \eqref{sigma},
$$\Sigma_{(F_1,...,F_k)}^{(F_1,...,F_k)}=\frac{\partial(F_1,...,F_k)}{\partial (x_1,...,x_n)}[g^{ij}](\frac{\partial(F_1,...,F_k)}{\partial (x_1,...,x_n)})^T.$$
By the same type of computation we have
\begin{eqnarray}
\Sigma_{(H_1,...,H_k)}^{(H_1,...,H_k)}&=&\frac{\partial(H_1,...,H_k)}{\partial (x_1,...,x_n)}[g^{ij}](\frac{\partial(H_1,...,H_k)}{\partial (x_1,...,x_n)})^T\nonumber \\
&=&\frac{\partial(h_1,...,h_k)}{\partial (F_1,...,F_k)}\frac{\partial(F_1,...,F_k)}{\partial (x_1,...,x_n)}[g^{ij}](\frac{\partial(F_1,...,F_k)}{\partial (x_1,...,x_n)})^T(\frac{\partial(h_1,...,h_k)}{\partial (F_1,...,F_k)})^T\nonumber \\
&=&\frac{\partial(h_1,...,h_k)}{\partial (F_1,...,F_k)}\Sigma_{(F_1,...,F_k)}^{(F_1,...,F_k)}(\frac{\partial(h_1,...,h_k)}{\partial (F_1,...,F_k)})^T.
\end{eqnarray}
We obtain the following equality
$$\det\Sigma_{(H_1,...,H_k)}^{(H_1,...,H_k)}(x)=\left(\det \frac{\partial(h_1,...,h_k)}{\partial (F_1,...,F_k)}(c)\right )^2\det\Sigma_{(F_1,...,F_k)}^{(F_1,...,F_k)}(x),$$
where $c\in \mathbb{R}^k$ with $c=(F_1(x),...,F_k(x))$. Consequently, on a regular leaf $L_c$ the standard control vector fields $\mathbf{v}_0^{(H_1,...,H_k)}$ generated by the set of conserved quantities $(H_1,...,H_k)$ and  respectively $\mathbf{v}_0^{(F_1,...,F_k)}$ generated by the set of conserved quantities $(F_1,...,F_k)$ differ by a constant. More precisely,
$$\mathbf{v}_{0|L_c}^{(H_1,...,H_k)}=\left(\det \frac{\partial(h_1,...,h_k)}{\partial (F_1,...,F_k)}(c)\right )^2 \mathbf{v}_{0|L_c}^{(F_1,...,F_k)}.$$

\section{Examples}
In this section we prove that the Landau-Lifschitz equation is a perturbed system which can be put in the form given by the equation \eqref{perturbat} with the perturbation being a standard control vector field.
This perturbation can also be formulated as a double bracket dissipation, see \cite{marsden-92}, \cite{bloch-krishnaprasad-marsden-ratiu}, \cite{ortega-planas}.
Also, we show that the cubic dissipation of the rigid body introduced by Morrison \cite{morrison-86} is again of the form given by a standard control vector field.  \vspace{2mm}

{\bf The Landau-Lifschitz equation}

One of the main objectives of the micromagnetics theory is to develop a formalism in which the macroscopic properties of a material can be simulated including the best approximation to the fundamental atomic behavior of the material.
The history of micromagnetics starts with a paper of Landau and Lifschitz, published in 1935, on the structure of a wall between two antiparallel domains.

The Landau-Lifschitz equation of motion for an individual spin has the form
\begin{equation}\label{L-L}
\dot{M}=\gamma M\times B+\frac{\lambda}{||M||^2}(M\times(M\times B)),
\end{equation}
where $M$ is the magnetization vector, $B$ is the magnetic field, $\gamma$ is the gyromagnetic ratio and $\lambda$ is the damping constant. Due to physical reasoning, we assume that the magnetic field $\gamma B$ is of potential type, i.e. $\gamma B=\nabla H$ for a smooth function $H:\mathbb{R}^3\diagdown \{(0,0,0)\}\rightarrow \mathbb{R}$. Also, we suppose that $\frac{\lambda}{\gamma}>0$. The phase-space of the problem is $\mathbb{R}^3\diagdown \{(0,0,0)\}$ endowed with the Lie-Poisson bracket given by the cross product.
An equivalent form of L-L equation the \eqref{L-L} is given by
\begin{equation}\label{LL}
\dot{M}=M\times \nabla H+\frac{\lambda}{\gamma ||M||^2}<M, \nabla H>M-\frac{\lambda}{\gamma} \nabla H.
\end{equation}

The unperturbed system $\dot{M}=M\times \nabla H$ conserves the Hamiltonian function $H$ and the Casimir function $C_0=\frac{1}{2} (M_1^2+M_2^2+M_3^2)$. We prove that the perturbation $$\frac{\lambda}{\gamma ||M||^2}<M, \nabla H>M-\frac{\lambda}{\gamma} \nabla H$$ is the standard control vector field $\mathbf{v}_0$ with the conserved function being $F=\sqrt{\frac{2\lambda}{\gamma}C_0}$ and the dissipated function being $G=-H$.

Using Remark \ref{T-expresie1F}, in the case of one conserved quantity, the symmetric contravariant 2-tensor $\mathbf{T}$ is given by
\begin{equation}\label{T3}
\mathbf{T}=-\nabla F\otimes \nabla F+||\nabla F||^2g^{-1},
\end{equation}
where $g$ is the Euclidean metric on $\mathbb{R}^3$. Using Theorem \ref{v0-cu-i} we obtain
\begin{eqnarray}\label{LLv0}
\mathbf{v_0}&=&\mathbf{i}_{dG}(\mathbf{T})=-<\nabla G,\nabla F>\nabla F+||\nabla F||^2\nabla G \nonumber \\
&=& <\nabla H,\sqrt{\frac{\lambda}{\gamma}}\frac{M}{||M||}>\sqrt{\frac{\lambda}{\gamma}}\frac{M}{||M||}-\frac{\lambda}{\gamma}\nabla H=\frac{\lambda}{\gamma ||M||^2}<M, \nabla H>M-\frac{\lambda}{\gamma} \nabla H.
\end{eqnarray}
We analyze the standard control vector field $\mathbf{v}_0$ given by the equation \eqref{LLv0} restricted on a regular leaf $L_c$.
In Cartesian coordinates we have the following expression of the symmetric contravariant  2-tensor field
$$\mathbf{T}=\frac{\lambda}{2\gamma C_0}(\sum_{i=1}^3(M_1^2+M_2^2+M_3^2-M_i^2)\frac{\partial}{\partial M_i}\otimes \frac{\partial}{\partial M_i}-\sum_{i,j=1,\,i\neq j}^3M_iM_j\frac{\partial}{\partial M_i}\otimes \frac{\partial}{\partial M_j})$$
Because $F$ is a constant of motion for the perturbed system \eqref{LL}, the regular leaves which are given by the spheres $L_c:=F^{-1}(c)$ are preserved by this perturbed dynamic. In spherical coordinates $(\theta, \varphi, r)$ we have
$M_1=r\sin\theta \cos\varphi,\,\,M_2=r\sin\theta \sin\varphi,\,\,M_3=r\cos\theta$. The symmetric contravariant  2-tensor field $\mathbf{T}(\theta,\varphi,r)$ becomes
$$\mathbf{T}=\frac{\lambda}{\gamma r^2}(\frac{\partial}{\partial \theta}\otimes \frac{\partial}{\partial \theta}+\frac{1}{\sin^2\theta}\frac{\partial}{\partial \varphi}\otimes \frac{\partial}{\partial \varphi})$$
and it is a degenerate tensor field.
If we choose a sphere $L_c$, where $r=\sqrt{\frac{\gamma}{\lambda}}c$, then $\mathbf{T}_{|L_c}$ becomes a nondegenerate symmetric contravariant  2-tensor field on $L_c$. We have the coordinate expression
$$\mathbf{T}_{|L_c}=\frac{\lambda^2}{\gamma ^2c^2}(\frac{\partial}{\partial \theta}\otimes \frac{\partial}{\partial \theta}+\frac{1}{\sin^2\theta}\frac{\partial}{\partial \varphi}\otimes \frac{\partial}{\partial \varphi}).$$
Consequently,
$$\tau_c=\frac{\gamma ^2c^2}{\lambda^2} (d\theta\otimes d\theta+\sin^2\theta d\varphi\otimes d\varphi).$$
According to Theorem \ref{caracterizare} $(iii)$, the standard control vector field on the sphere $L_c$  is a gradient vector field and it has the expression
$$\mathbf{v_0}_{|L_c}=-\frac{\lambda^2}{\gamma ^2c^2}\frac{\partial H_{|L_c}}{\partial \theta}\frac{\partial}{\partial \theta}-\frac{\lambda^2}{\gamma ^2c^2\sin^2\theta}\frac{\partial H_{|L_c}}{\partial \varphi}\frac{\partial}{\partial \varphi}.$$
\begin{rem}
The induced metric on the sphere $L_c$ is given by
$$i_c^*g=\frac{\gamma}{\lambda}c^2  (d\theta\otimes d\theta+\sin^2\theta d\varphi\otimes d\varphi)$$
and consequently, we have
$$\frac{1}{||\nabla F||^2\circ i_c}i_c^*g=\frac{\gamma ^2c^2}{\lambda^2} (d\theta\otimes d\theta+\sin^2\theta d\varphi\otimes d\varphi).$$
As stated in Theorem \ref{caracterizare} $(i)$, we obtain the equality $\tau_c=\frac{1}{||\nabla F||^2\circ i_c}i_c^*g$.
\begin{flushright}
$\bigtriangleup$
\end{flushright}
\end{rem}

We can also write the standard control vector field $\mathbf{v_0}$ by using the orthogonal projector defined by the formula \eqref{proiector}. More precisely,
$$P_{T_xL_c}\nabla G=\frac{1}{||\nabla F||^2}\det \left(%
\begin{array}{cc}
  <\nabla F,\nabla F> & <\nabla G,\nabla F> \\
\nabla F & \nabla G \\
\end{array}%
\right)
=\nabla G-\frac{1}{||\nabla F||^2}<\nabla F,\nabla G>\nabla F.$$
By using Remark \ref{T-expresie1F} and Theorem \ref{v0-proiector}, we obtain
$$\mathbf{v_0}=\frac{\lambda}{\gamma ||M||^2}<M, \nabla H>M-\frac{\lambda}{\gamma} \nabla H.$$

The Hamiltonian function $H$ decreases along the solutions of the Landau-Lifschitz equation as $G$ is an increasing function along these solutions, see Theorem \ref{crescator}.
\vspace{2mm}

{\bf The metriplectic dissipation of the rigid body}

The motion of a rigid body can be reduced to the translation of
the center of mass and the rotation about the center of mass. Rotation is conveniently
described in a coordinate system with the origin at the center of
mass and the axes along the principal central axes of inertia, by
Euler's equations. These equations can be written in the following
form
$$\left\{%
\begin{array}{ll}
    \dot{x_1}=(\frac{1}{I_3}-\frac{1}{I_2})x_2x_3+u_1 \\
    \dot{x_2}=(\frac{1}{I_1}-\frac{1}{I_3})x_1x_3+u_2\\
    \dot{x_3}=(\frac{1}{I_2}-\frac{1}{I_1})x_1x_2+u_3\\
\end{array}%
\right.$$ where
$x_1=I_1\omega_1,\,x_2=I_2\omega_2,\,x_3=I_3\omega_3$ are the
components of the angular momentum vector, and $I_1>I_2>I_3$ are the principal moments of
inertia, and $\omega_1,\omega_2,\omega_3$ are the components of the
angular velocity and $u_1,u_2,u_3$ are the components of applied
torques.
The system of free rotations has the well known Hamilton-Poisson formulation $(so(3)^*,\Pi_{-},H),$
where $\Pi_{-}$ is the minus Lie-Poisson structure on
$so(3)^*$ and the Hamiltonian function is given by
$H(x_1,x_2,x_3)=\frac{1}{2}(\frac{x_1^2}{I_1}+\frac{x_2^2}{I_2}+\frac{x_3^2}{I_3})$, see \cite{marsden-ratiu}.
In \cite{morrison-86}, Morrison has introduced the following class of metriplectic dissipation for the rigid body
$${\bf u}=[h^{ij}]\nabla C,$$
where $C$ is a Casimir of $\Pi_{-}$ and
$$[h^{ij}](x)=\left(%
\begin{array}{ccc}
  \frac{x_2^2}{I_2^2}+\frac{x_3^2}{I_3^2} & -\frac{x_1 x_2}{I_1 I_2} & -\frac{x_1 x_3}{I_1 I_3} \\
  -\frac{x_1 x_2}{I_1 I_2} & \frac{x_1^2}{I_1^2}+\frac{x_3^2}{I_3^2} & -\frac{x_2 x_3}{I_2 I_3} \\
  -\frac{x_1 x_3}{I_1 I_3} & -\frac{x_2 x_3}{I_2 I_3} & \frac{x_1^2}{I_1^2}+\frac{x_2^2}{I_2^2} \\
\end{array}%
\right).$$
Next, we prove that the above control vector field ${\bf u}$ is the standard control vector field generated by the conserved function $F=H$ and the dissipated function $G=C$.
As before, by using the expression \eqref{T-cu1F} of the tensor $\mathbf{T}$ and Theorem \ref{v0-cu-i} we have for this case
$${\bf v}_0={\bf i}_{dG}(\mathbf{T})=(-\nabla H\otimes\nabla H+||\nabla H||^2\mathbb{I})\nabla C=[h^{ij}]\nabla C.$$
The above form of the dissipation vector field was also found in \cite{birtea-comanescu}. For a fixed $x\in \mathbb{R}^3$, the matrix associated with the linear operator $||\nabla H(x)||^2P_{T_xL_c}$ is the matrix $[h^{ij}](x)$.

In what follows we compute the standard control vector field ${\bf v_0}$ on the regular leaves $L_c:=H^{-1}(c)$. A system of local coordinates $(\theta,\varphi,r)$ for the ellipsoid $L_c$ is given by
$$x_1=r\sqrt{I_1}\sin\theta\cos\varphi,\,\,x_2=r\sqrt{I_2}\sin\theta\sin\varphi,\,\,x_3=r\sqrt{I_3}\cos\theta,$$
where $r=\sqrt{2c}$. For this choice of parameters we have
$$||\nabla H||^2=2c(\frac{\sin^2\theta\cos^2\varphi}{I_1}+\frac{\sin^2\theta\sin^2\varphi}{I_2}+\frac{\cos^2\theta}{I_3})$$
and the induced metric
\begin{eqnarray*}
i^*_cg &=&2c(I_1\cos^2\theta\cos^2\varphi+I_2\cos^2\theta\sin^2\varphi+I_3\sin^2\theta)d\theta\otimes d\theta\nonumber \\
& &+2c(I_2-I_1)\sin\theta\cos\theta\sin\varphi\cos\varphi(d\theta\otimes d\varphi+d\varphi\otimes d\theta)
\nonumber \\
& &+2c\sin^2\theta(I_1\sin^2\varphi+I_2\cos^2\varphi) d\varphi\otimes d\varphi.
\end{eqnarray*}
For the standard Casimir function $C_0=\frac{1}{2}(x_1^2+x_2^2+x_3^2)$ and by Theorem \ref{caracterizare} $(iii)$, the standard control vector field on the ellipsoid $L_c$  is a gradient vector field which has the expression
$${\bf v_0}_{|L_c}=2c\sin\theta\cos\theta(\frac{1}{I_3}-\frac{\sin^2\varphi}{I_2}-\frac{\cos^2\varphi}{I_1})
\frac{\partial}{\partial \theta}+2c(\frac{1}{I_1}-\frac{1}{I_2})\sin\varphi\cos\varphi\frac{\partial}{\partial \varphi}.$$
We observe that for the axisymmetric case $I_1=I_2$, the dissipation takes place only in the angle coordinate $\theta$,
$${\bf v_0}_{|L_c}=2c\sin\theta\cos\theta(\frac{1}{I_3}-\frac{1}{I_2})
\frac{\partial}{\partial \theta}.$$

\section{Appendix}

Let $(M,g)$ be a Riemannian manifold. We will recall some standard formulas that we have used throughout this paper.

\noindent {\bf Ricci symbol}
\begin{equation}\label{Ricci}
    \epsilon_{i_1...i_r}=\left\{
                           \begin{array}{ll}
                             1, & \hbox{if $i_1,...,i_r$ are distinct, and they are an even permutation of $\{1,...,r\}$;} \\
                             -1, & \hbox{if $i_1,...,i_r$ are distinct, and they are an odd permutation of $\{1,...,r\}$;} \\
                             0, & \hbox{otherwise.}
                           \end{array}
                         \right.
\end{equation}
From the definition of the Ricci symbol, we have
\begin{equation}\label{q-sarit}
    \epsilon_{i_1...i_{r-1}q}=(-1)^{r-k-1}\epsilon_{i_1...i_kqi_{k+1}...i_{r-1}}
\end{equation}
{\bf Generalized Kroneker $\delta$-symbol}
\begin{equation}\label{Kroneker}
    \delta_{i_1...i_r}^{j_1...j_r}=\left\{
                           \begin{array}{ll}
                             1, & \hbox{if $i_1,...,i_r$ are distinct, and $\{j_1,...,j_r\}$ is an even permutation of $\{i_1,...,i_r\}$;} \\
                             -1, & \hbox{if $i_1,...,i_r$ are distinct, and$\{j_1,...,j_r\}$ is an odd permutation of $\{i_1,...,i_r\}$;} \\
                             0, & \hbox{otherwise.}
                           \end{array}
                         \right.
\end{equation}
Using the two definitions above, we obtain the equality
\begin{equation}\label{Ricci-Kroneker}
    \delta_{i_1...i_r}^{j_1...j_r}=\epsilon_{i_1...i_r}\epsilon_{j_1...j_r}.
\end{equation}
For a fixed $q$ we have
\begin{equation}\label{q-constant}
    \delta_{i_1...i_{r-1}q}^{j_1...j_{r-1}q}=\delta_{i_1...i_{r-1}}^{j_1...j_{r-1}}.
\end{equation}
If we do the summation after the index $q$ from $1$ to $n$, we obtain the formula
\begin{equation}\label{q-sumare}
    \delta_{i_1...i_{r-1}q}^{j_1...j_{r-1}q}=(n-r+1)\delta_{i_1...i_{r-1}}^{j_1...j_{r-1}},
\end{equation}
and more generally
\begin{equation}\label{q-sumare-factorial}
    \delta_{i_1...i_ri_{r+1}...i_p}^{j_1...j_ri_{r+1}...i_p}=\frac{(n-r)!}{(n-p)!}\delta_{i_1...i_r}^{j_1...j_r}.
\end{equation}
The determinant of an $r\times r$ matrix can be written using the formulas
\begin{equation}\label{determinant}
   \det\left(
         \begin{array}{ccc}
           a_{11} & ... & a_{1r} \\
           ... & ... & ... \\
           a_{r1} & ... & a_{rr} \\
         \end{array}
       \right)=\epsilon_{i_1...i_r}a_{1i_1}a_{2i_2}...a_{ri_r}
\end{equation}
and
\begin{equation}\label{determinant-amestecat}
    \epsilon_{j_1...j_r}\det\left(
         \begin{array}{ccc}
           a_{11} & ... & a_{1r} \\
           ... & ... & ... \\
           a_{r1} & ... & a_{rr} \\
         \end{array}
       \right)=a_{j_1i_1}a_{j_2i_2}...a_{j_ri_r}\epsilon_{i_1...i_r}.
\end{equation}
Using the notation \eqref{sigma} and developing the determinant, we obtain
\begin{equation}\label{sigma-developare}
    \det\Sigma_{(F_1,...,F_k)}^{(F_1,...,F_k)}=\frac{\partial F_1}{\partial x^{b_1}}...\frac{\partial F_k}{\partial x^{b_k}}g^{b_1s_1}...g^{b_ks_k}\det\left(
                                     \begin{array}{ccc}
                                       \frac{\partial F_1}{\partial x^{s_1}} & ... & \frac{\partial F_k}{\partial x^{s_1}} \\
                                       ... & ... & ... \\
                                       \frac{\partial F_1}{\partial x^{s_k}} & ... & \frac{\partial F_k}{\partial x^{s_k}} \\
                                     \end{array}
                                   \right).
\end{equation}
The formula for the Hodge star operator $\ast:\Omega^r(M)\rightarrow \Omega^{n-r}(M)$ in local coordinates is given by
\begin{equation}\label{Hodge-formula}
    \ast(dx^{i_1}\wedge ...\wedge dx^{i_r})=\frac{\sqrt{|g|}}{(n-r)!}g^{i_1j_1}...g^{i_rj_r}\epsilon_{j_1...j_rj_{r+1}...j_n}dx^{j_{r+1}}\wedge ...\wedge dx^{j_n},
\end{equation}
where $|g|$ is the determinant of the symmetric matrix associated to the Riemannian metric $g$.

\end{document}